\documentclass[11pt, reqno]{amsart}
\usepackage{amssymb}
\usepackage{eucal}
\newcommand{\R}{\mathbb R}

\newcommand{\s}{{\mathcal S}}
\newcommand{\Z}{\mathbb Z}

\newcommand{\pd}{\partial}

\newcommand{\tres}{|\!|\!|}

\newcommand{\secao}[1]{\section{#1}\setcounter{equation}{0}}

\newtheorem{theorem}{Theorem}[section]

\newtheorem{remark}{Remark}[section]
\newtheorem{lemma}[theorem]{Lemma}

\begin{document}

%\today{}

\title[On uniqueness and decay]{On uniqueness and decay of solution for Hirota equation}

\author{X. Carvajal}
\address{Instituto de Matem\'atica, UFRJ, Rio de Janeiro,   Brazil}
\email{carvajal@im.ufrj.br}
\thanks{X. C. was partially supported by CNPq, Brazil.}

\author{M. Panthee}
\address{Centro de Matem\'atica, Universidade do Minho, 4710-057, Braga, Portugal. }
\email{mpanthee@math.uminho.pt}

\thanks{M. P.  was partially supported by the Research Center of 
Mathematics of the University of Minho, Portugal through the FCT Pluriannual Funding Program.}

\keywords{Schr\"{o}dinger equation, Korteweg-de Vries equation,
\hfill\break\indent smooth solution,  unique continuation property, compact support}

\begin{abstract}
We address the question of the uniqueness of solution to the initial value problem associated to the equation
\begin{equation*}
\partial_{t}u+i\alpha
\partial^{2}_{x}u+\beta  \partial^{3}_{x}u+i\gamma|u|^{2}u+\delta
|u|^{2}\partial_{x}u+\epsilon u^{2}\partial_{x}\overline{u}  =  0,
 \quad x,t \in \R,
\end{equation*}
and prove that a certain decay property of  the difference $u_1-u_2$ of two solutions $u_1$ and $u_2$ at two different instants of times $t=0$ and $t=1$, is sufficient to ensure that $u_1=u_2$ for all the time.
\end{abstract}
\maketitle
%\noindent

%\vglue-2truein

\section{Introduction}

In this  work we consider the  following equation
\begin{equation}\label{0y0}
\partial_{t}u+i\alpha
\partial^{2}_{x}u+\beta  \partial^{3}_{x}u+i\gamma|u|^{2}u+\delta
|u|^{2}\partial_{x}u+\epsilon u^{2}\partial_{x}\overline{u}  =  0,
 \quad x,t \in \R,
\end{equation}
where $\alpha,\beta\in \R $, $\beta \ne 0$, $\gamma, \delta, \epsilon
  \in \mathbb{C}$ and  $u = u(x, t)$ is a complex valued function. Our main concern is to find a decay property satisfied by the difference of two different solutions at two different instants of time that is sufficient to prove the uniqueness of the solution to the initial value problem (IVP) associated to \eqref{0y0}.

The equation \eqref{0y0}, with the mixed structure of the Korteweg-de
 Vries (KdV) and the Schr\"{o}dinger equations, was proposed by Hasegawa and Kodama in \cite{[H-K], [Ko]}
 to describe the nonlinear propagation of pulses in optical fibers. This equation is also known as Hirota equation in the literature.  Several aspects of this equation including well-posedness issues, solitary wave solutions, unique continuation property,  have been studied by various authors recently, see for example \cite{CL1}, \cite{CP1},  \cite{CP2}, \cite{[C2]}, \cite{[G]}
 and references therein.

 Study of the unique continuation property (UCP) for certain models has drawn much attention of a considerable section of mathematicians in recent time, see for example \cite{[B2]},
  \cite{CP1}, \cite{IR-1} --  \cite{KRS.1}, \cite{SM} -- \cite{[S-Sch]}, \cite{DT}, \cite{[Z]} and references therein.
  In particular,  in \cite{CP1} and \cite{CP2} we addressed  the UCP for the equation (\ref{0y0}).
   In \cite{CP1}, we proved that if a sufficiently smooth solution $u$ to the initial value problem associated to (\ref{0y0})
    is supported in a half line at two different instants of time then $ u$ vanishes identically.
    The precise statement of our result in \cite{CP1} is the following.
    
\begin{theorem}\cite{CP1}.\label{teo12}
Let $u \in {C}([t_{1},t_{2}];
H^{s})\cap {C}^{1}([t_{1},t_{2}]; H^{1})$, $s
\geq 4$ be a strong solution of the  equation (\ref{0y0})  with $\alpha, \beta,
 \gamma , \delta , \epsilon \in \R$, $\beta \ne 0$. If there exists
 $t_{1} < t_{2} $ such that
\begin{eqnarray} {\rm{supp}}\, u(\cdot,t_{j}) & \subset & (-\infty, a),
 \qquad j=1,2 \\
or, \, (\, {\rm{supp}} \,
u(\cdot,t_{j}) & \subset & (b,\infty), \qquad j=1,2 \, ).
 \end{eqnarray}
Then $ u(t)= 0$ for all $t\in [t_1, t_2].$
\end{theorem}

In our subsequent work \cite{CP2}, we obtained more general uniqueness property for solution of the IVP associated to \eqref{0y0}.

\begin{theorem}\cite{CP2}.\label{ucp2} Let $u, v \in {C}([t_{1},t_{2}];
H^{s})\cap {C}^{1}([t_{1},t_{2}]; H^{1})$, $s \geq 4$ be strong solutions of
the equation (\ref{0y0}) with $\alpha, \beta, \gamma , \delta ,
\epsilon \in \R$, $\beta \ne 0$. If there exists $b \in \R$ such that
\begin{eqnarray} u(x,t) & = & v(x, t), \qquad (x, t)
 \in (b, \infty)\times \{t_1, t_2\}, \label{supp1} \\ 
or, \, ( u(x,t) & = & v(x, t), \qquad (x, t)
\in (-\infty, b)\times \{t_1, t_2\}).\label{supp0}  \end{eqnarray} 
Then \[ u(t)= v(t) \quad \forall \,t\in [t_1, t_2] .\]  
\end{theorem}

\begin{remark}Theorem \ref{teo12} is the special case of Theorem \ref{ucp2} when $v\equiv 0$.
\end{remark}

Motivation to obtain the above results is the following observation. Consider the IVP associated to the linear part of (\ref{0y0}), i.e.,
\begin{equation}\label{lnpt1}
\begin{cases}
u_t + i\alpha u_{xx} + \beta u_{xxx} = 0,\\
u(x, 0) = u_0(x).
\end{cases}
\end{equation}
If $u$ and $v$ are solutions to (\ref{lnpt1}) then $w:= u-v$ is also a
solution to (\ref{lnpt1}) with initial data $w(x,0)= u(x,0)-v(x,0):=
w_0(x)$. If $w_0$ is sufficiently smooth and has compact support, then
using the Paley-Wiener theorem it is easy to see (for detail see
\cite{CP1}) that $w \equiv 0$, i.e., $u \equiv v$. But the proof of the same property is not so simple when one considers the nonlinear terms as well, because in this case $w:= u-v$ is no more a solution. To overcome this situation, we generalized and employed the techniques developed in the context of the generalized KdV equation by Kenig-Ponce-Vega in \cite{[KPV3]} and  \cite{KPV-UC2} to prove Theorems \ref{teo12} and \ref{ucp2} .

 Quite recently, Escauriaza, Kenig, Ponce and Vega in \cite{E-KPV1} introduced a new
technique to obtain sufficient conditions on the behavior of the difference $u_1-u_2$ of two solutions $u_1$ and $u_2$ of the generalized KdV equation at two different instants of time $t=0$ and $t=1$ that guarantees $u_1 \equiv u_2$.  In \cite{E-KPV1}, the authors obtained a sharp decay condition to guarantee the uniqueness of solution to the generalized KdV equation. So, there arise a natural question, whether one can find such a decay condition to get uniqueness property for a mixed equation of the KdV and Schr\"odinger type.
In this work, we shall extend the approach in \cite{E-KPV1} to address this question to the IVP  associated to the Hirota equation (\ref{0y0}) which has a mixed structure of the KdV and the Schr\"odinger  equations.    Our first main result of this work is the following.

\begin{theorem}\label{ucp3} Let $u_1, u_2\in C([0,1];
H^{3}(\R))\cap L^2(|x|^2dx))$,  be strong solutions of
the equation (\ref{0y0}) with $\alpha, \beta, \gamma , \delta ,
\epsilon \in \R$, $\beta \ne 0$. If, for any $a>0$,
\begin{equation} u_1(\cdot, 0)-u_2(\cdot, 0), \quad u_1(\cdot, 1)-u_2(\cdot, 1) \in H^1(e^{ax_+^{3/2}}dx),
\end{equation}
then \[ u_1 \equiv  u_2.\]
\end{theorem}

To prove Theorem \ref{ucp3} we follow  the techniques introduced in
\cite{E-KPV1} by deriving some new estimates that are appropriate to work with the structure of the equation under consideration. Although the idea and estimates are similar to the ones introduced in \cite{E-KPV1}, the presence of the Schr\"odinger term in the linear part creates obstacle to obtain such estimates, which can be seen more explicitly in the derivation of the lower estimates in Section \ref{prel}. The proofs of several estimates that are crucial to prove the main results depend on the estimates obtained on our previous works \cite{CP1} and \cite{CP2}, where the exponential decay property of the solution was necessary. As observed in \cite{CP1} and \cite{CP2}, the presence of the third order derivative in
(\ref{0y0}) is fundamental to obtain the desired exponential decay property of the
solution. So we will suppose $\beta \neq 0$ throughout this work. To be more precise, let us recall the
following remark from \cite{CP1}.

\begin{remark}\label{obspat} We can suppose $\beta>0$.  In fact, for
$\alpha \neq 0$ we can suppose $\beta =|\alpha|/3$.

If $\beta <0$ we define $w(x,t)=u(-x,t)$ then $w$ is a solution to the
equation (\ref{0y0}) with the coefficient of the third derivative is positive.

If $\beta >0$ and $\alpha \neq 0$ we define
$w(x,t)=u(\tilde{a}^{-1}x,t)$ with $\tilde{a}=|\alpha| /3\beta$, then
$w$ is a solution of the equation \begin{equation} w_{t}+i\alpha
\tilde{a}^{2} w_{xx}+\beta \tilde{a}^{3}
w_{xxx}+i\gamma|w|^{2}w+\delta \tilde{a}|w|^{2}w_{x}+\epsilon
\tilde{a}w^{2}\bar{w}_x=0, \nonumber
\end{equation}
 and we have $\beta \tilde{a}^{3}=|\alpha| \tilde{a}^{2} /3$.
\end{remark}

As mentioned earlier, we are interested in finding a decay condition satisfied by the difference of two solutions at two different instants of time $t=0$ and $t=1$ that is sufficient to get the uniqueness of solution to the IVP associated to \eqref{0y0}. Note that, while treating with the difference of two solutions, we need to address an equation with variable coefficients (see \eqref{mx100} below). Therefore, in the first instant, we consider a more general equation,
\begin{equation}\label{mx00}
w_t+i \alpha w_{xx}+\beta w_{xxx}+ a_2(x,t)w_{xx}+a_1(x,t)w_{x}+b_1(x,t)\bar{w}_x+a_0(x,t)w+b_0(x,t)\bar{w}=0,
\end{equation}
and prove the following result.

\begin{theorem}\label{ucp5} 
Assume that the coefficients in \eqref{mx00} satisfy that
\begin{equation}\label{mx01}
\begin{cases}
a_0, b_0 \in L_{xt}^{4/3}\cap L_{x}^{16/13}L_{t}^{16/9}\cap L_{x}^{8/7}L_{t}^{8/3},\\
a_1, b_1 \in  L_{x}^{16/13}L_{t}^{16/9}\cap L_{x}^{8/7}L_{t}^{8/3}\cap L_{x}^{16/15}L_{t}^{16/3}, \\
a_2 \in L_{x}^{8/7}L_{t}^{8/3}\cap L_{x}^{16/15}L_{t}^{16/3}\cap L_{x}^{1}L_{t}^{\infty}.
\end{cases}
\end{equation}
and
\begin{equation}\label{mx02}
\begin{cases}
a_0, b_0, a_1, b_1, a_2, (a_0)_x, (b_0)_x, (a_1)_x, (b_1)_x, (a_2)_x, (a_2)_{xx}, (a_2)_{xxx}, (a_2)_{t} \in L^{\infty}(\R \times [0,1]),\\
a_2, (a_2)_{t} \in L_t^{\infty}([0,1]; L_x^{1}(\R )).
\end{cases}
\end{equation}
If $w \in C\left([0,1]; H^{2}(\R))\cap L^2(|x|^2dx) \right)$is a strong solution of \eqref{mx00} with
$$
w(\cdot, 0), w(\cdot, 1) \in H^{1}(e^{ax_+^{3/2}}dx), \quad \forall a>0,
$$
then $w\equiv0$.
\end{theorem}

Once we get this theorem, the proof of the main theorem follows by proving that the variable coefficients involved in the equation in question satisfy the respective estimates.

%%%%%%%%%%%%%%%%%%%%%%%%%%%%%%%%%%%%%%%%%%%%%%
%%%%%%%%%%%%%%%%%%%%%%%%%%%%%%%%%%%%%%%%%%%%%%

Our next result is concerned with the existence of solution to the IVP associated to \eqref{0y0} that decays asymptotically in $x$. First, let us consider the IVP \eqref{lnpt1} associated to the linear part of \eqref{0y0}. The solution to the IVP \eqref{lnpt1} is given by,
\begin{equation}\label{lin-sol}
u(x,t) = \frac1{\sqrt[3]{3t}}G\left(\frac{\cdot}{\sqrt[3]{3t}}\right)*u_0(x),
\end{equation}
where
\begin{equation}
G(x)=\int_{\R}e^{\frac{i8\pi^3}{3}\eta^3+\frac{i\alpha t^{1/3}4\pi^2}{\sqrt[3]{9}}\eta^2+2\pi i \eta x}d\eta.
\end{equation}

With some easy calculations, one can obtain 
\begin{equation}\label{lp-8}
|G(x)|= |Ai(x-4\pi^2B^2)|,
\end{equation}
where $Ai$ is the usual Airy function given by
$$
Ai(x)= \int_{\R}e^{2\pi i x\xi+\frac83i\pi^3\xi^3}d\xi,
$$
and
$$B= \frac{\alpha t^{1/3}}{2\sqrt[3]{9}\pi}.
$$

If $x>8\pi^2B^2$, we get
\begin{equation}\label{air-25}
|Ai(x-4\pi^2B^2)| \leq C e^{-C(x-4\pi^2B^2)^{3/2}}\leq Ce^{-\frac{C}{2^{3/2}}x^{3/2}}.
\end{equation}

Therefore, from \eqref{lp-8} and \eqref{air-25} we have, for any $t\in [0, 1]$,
\begin{equation}\label{lp-9}
|G(x)| \leq Ce^{-\frac{C}{2^{3/2}}x^{3/2}}, 
\end{equation}
provided, $x>\frac{2\alpha^2}{3\sqrt[3]{3}}$.

The estimate \eqref{lp-9} shows that the decay condition in Theorem \ref{ucp3} is in accordance with the decay of the function $G$ that describes the solution of the linear part.

In what follows, we show the existence of a local solution to the IVP associated to \eqref{0y0}  that satisfies the similar decay property as the linear solution described above. More precisely, our second main theorem reads as follows.

\begin{theorem}\label{teoairy}
There exists $u_0 \in \mathbb{S}(\R)$, $u_0 \neq 0$ and $T>0$ such that the  IVP associated to \eqref{0y0} with data $u_0$ has a solution $u \in C([0,T]:\mathbb{S}(\R))$ which satisfies
$$
|u(x,t)| \le c e^{-x^{3/2}/3}, \quad x>1,\,\,t\in [0,T],
$$
for some constant $c>0$.
\end{theorem}

We organize this article in the following manner. In Sections \ref{prel-1} and \ref{prel} we prove some preliminary estimates (upper estimate and lower estimate) which play a vital role
to prove our main theorem. In Section \ref{main-results} we present a proof of a more general result, Theorem \ref{ucp5}, and then the proofs of the main results of this work,
Theorem \ref{ucp3} and Theorem \ref{teoairy}. Before leaving this section, let us record
some notations that are used throughout this work.\\

%%%%%%%%%%%%%%%%%%%%%%%%%%%%%%%%%%%%%%%%%%%%%
%%%%%%%%%%%%%%%%%%%%%%%%%%%%%%%%%%%%%%%%%%%%%

\noindent
{\bf{Notations:}} We use $\hat{f}(\xi)$ and $\hat{f}(\xi, \tau)$ to denote the Fourier transform defined by $\hat{f}(\xi) = \frac{1}{\sqrt{2\pi}} \int
e^{-ix\xi}f(x)\,dx,$ and $\hat{f}(\xi,\tau) = \frac{1}{{2\pi}} \int
e^{-i(x\xi+t\tau)}f(x,t)\,dxdt$ respectively. We use $L_x^pL_t^q$ to denote mixed Lebesgue spaces. We write $A\lesssim B$
if there exists a constant $c >0$ such that $ A \leq cB$.

%%%%%%%%%%%%%%%%%%%%%%%%%%%%%%%%%%%%%%%%%%%%%
\secao{Upper  estimates}\label{prel-1}
%%%%%%%%%%%%%%%%%%%%%%%%%%%%%%%%%%%%%%%%%%%%%

This section is devoted to prove upper estimates that play crucial role in the proof of the main results. Let us first define the following operators
\begin{equation}\label{def-H}
Hf = (\partial_t +i\alpha \partial_x^2 +\beta \partial_x^3)f, \qquad H_mf=(\partial_t +e^{mx}(i\alpha \partial_x^2 +\beta \partial_x^3)e^{-mx})f.
\end{equation}
By Remark \ref{obspat}, we can suppose that $\beta>0$ and $|\alpha|/3=\beta$. Also, let us define
 $v:= e^{mx}u$, where $u$ is a solution to \eqref{0y0}. We begin with the following result.
 \begin{lemma}\label{lema-2}
The following estimate holds
\begin{equation}\label{eq2.8}
\|v\|_{L_{t}^{\infty}L_x^{2}} \leq C \Big( \|v(\cdot, 0)\|_{L^2}+\|v(\cdot, 1)\|_{L^2}\Big) + C\|H_mv\|_{L_{t}^{1}L_x^{2}}.
\end{equation}
\end{lemma}
\begin{proof}
We have
\begin{equation}\label{eq2.9}
H_m=\partial_t +e^{mx}(i\alpha \partial_x^2 +\beta \partial_x^3)e^{-mx}
= \partial_t +i\alpha(e^{mx} \partial_xe^{-mx})^2 +\beta (e^{mx} \partial_xe^{-mx})^3.
\end{equation}

Also using  $(e^{mx} \partial_xe^{-mx})^j= (\partial_x-m)^j$, $j=1,2,3$, we obtain
\begin{equation}\label{eq2.10}
H_mf= \partial_t +\beta\partial_x^3 +(i\alpha -3\beta m)\partial_x^2 +(3\beta m^2 -2i\alpha m)\partial_x +i\alpha m^2 -\beta m^3)f.
\end{equation}

The symbol of $H_m$ is given by
\begin{equation}\label{eq2.11}\begin{split}
 &i\tau -i\beta\xi^3 -(i\alpha -3\beta m)\xi^2 +(3\beta m^2 -2i\alpha m)i\xi +i\alpha m^2 -\beta m^3\\
 &= i(\tau -\beta \xi^3 -\alpha \xi^2 +3\beta m^2\xi +\alpha m^2)-(\beta m^3-2\alpha m\xi -3\beta m\xi^2).
 \end{split}
\end{equation}
Note that the real part of the symbol vanishes at
\begin{equation}\label{eq2.12}
\xi_{\pm} = \frac{-\alpha\pm \sqrt{\alpha^2 +3\beta^2m^2}}{3\beta}.
\end{equation}

As noted in \cite{E-KPV1}, by an approximation argument, it suffices to prove \eqref{eq2.8} for $v\in C^{\infty}([0, 1]; \s(\R))$ with $\hat{v}(\xi, t)=0$ near $\xi_{\pm}$ for all $t\in [0, 1]$.

Now, consider $f \in C^{\infty}([0, 1]; \s(\R))$ with $f(x,t) =0$ for $t$ near $0$ and $1$ so that we can extend $f$ as zero outside the strip $\R\times [0, 1]$. Also suppose that $\hat{f}(\xi, t) =0$ for $\xi$ near $\xi_{\pm}$ for all $t\in \R$. For such a function $f$, define an operator $T$ by
\begin{equation}\label{eq2.13}
\widehat{Tf}(\xi, \tau) := \frac{\hat{f}(\xi, \tau)}{i(\tau -\beta \xi^3 -\alpha \xi^2 +3\beta m^2\xi +\alpha m^2)-(\beta m^3-2\alpha m\xi -3\beta m\xi^2)}.
\end{equation}
We claim that the operator $T$ satisfies  the estimate
\begin{equation}\label{eq2.14}
\|Tf\|_{L_{t}^{\infty}L_x^2} \leq C\|f\|_{L_{t}^{1}L_x^2},
\end{equation}
which in turn implies \eqref{eq2.8}.

To prove this, let us define $\eta_{\varepsilon} \in C^{\infty}(\R)$, $\varepsilon \in (0, \frac14)$ such that
$$
\eta_{\varepsilon}(t) =1, \quad t\in [2\varepsilon, 1-2\varepsilon]; \qquad {\rm supp}\, \eta_{\varepsilon} \subset [\varepsilon, 1-\varepsilon].
$$

Define
$$v_{\varepsilon} =\eta_{\varepsilon}(t)v(x,t), \qquad f_{\varepsilon}(x,t) = H_m(v_{\varepsilon})(x,t),$$
then, $v_{\varepsilon} =Tf_{\varepsilon}$. Now \eqref{eq2.14} gives,
\begin{equation}\label{eq2.16}
\|v_{\varepsilon}\|_{L_{t}^{\infty}L_x^2} \leq C\|H_m(v_{\varepsilon})\|_{L_{t}^{1}L_x^2}\leq C\|\eta_{\varepsilon}´(t)v\|_{L_{t}^{1}L_x^2}+\|\eta_{\varepsilon}H_m(v)\|_{L_{t}^{1}L_x^2}.
\end{equation}

Letting $\varepsilon \to 0$, the left hand side of \eqref{eq2.16} converges to $\|v\|_{L_{[0, 1]}^{\infty}L_x^2}$ and the limit in the right hand side is bounded by
$$C\big(\|v(\cdot, 0)\|_{L^2} +\|v(\cdot, 1)\|_{L^2}\big) +C\|H_mv\|_{L_{t}^{1}L_x^2}.
$$

Therefore, our task is to prove \eqref{eq2.14}. As noted in \cite{E-KPV1},  it is enough to prove that for $f(x,t) =f(x)\otimes\delta_{t_0}(t)$, with $\hat{f}(\xi) =0$ near $\xi_{\pm}$, with $t_0 \in (0, 1)$, one has
\begin{equation}\label{eq2.17}
\|Tf\|_{L_{t}^{\infty}L_x^2} \leq C\|f\|_{L^2},
\end{equation}
where $C$ is independent of $t_0$.

Let us recall the formulas
\begin{equation}\label{eq2.18}
\Big(\frac{1}{\tau+ib}\Big)^{\vee} (t)= C \begin{cases} \chi_{(-\infty, 0)}(t)e^{tb}, \quad b>0\\
                                                                                                           \chi_{(0, \infty)}(t)e^{tb}, \quad b<0,
                                                                                 \end{cases}
\end{equation}
so that for $a, b \in \R$,
\begin{equation}\label{eq2.19}
\Big(\frac{e^{it_0\tau}}{\tau-a+ib}\Big)^{\vee} (t)
= C e^{ita} \begin{cases} \chi_{(-\infty, 0)}(t-t_0)e^{(t-t_0)b}, \quad b>0\\
                                                \chi_{(0, \infty)}(t-t_0)e^{(t-t_0)b}, \quad b<0.
                                                                                 \end{cases}
\end{equation}

Hence,
\begin{equation}\label{eq2.20}
\begin{split}
\widehat{Tf}(\xi, \tau) &:= \frac{e^{it_0\tau}\hat{f}(\xi)}{i\{(\tau -\beta \xi^3 -\alpha \xi^2 +3\beta m^2\xi +\alpha m^2)+i(\beta m^3-2\alpha m\xi -3\beta m\xi^2)\}}\\
& = -i \frac{e^{it_0\tau}\hat{f}(\xi)}{\tau -a(\xi)+ib(\xi)}.
\end{split}
\end{equation}

Combining \eqref{eq2.19} and \eqref{eq2.20}, it is clear that the operator $T$ acting on these functions becomes the one variable operator $R$ given by,
\begin{equation}\label{eq2.21}
\begin{split}
\widehat{Rf}(\xi) &= \big(\chi_{\{b(\xi)>0\}}(\xi)e^{ita(\xi)}e^{(t-t_0)b(\xi)}\chi_{(-\infty, 0)}(t-t_0) \big) \hat{f}(\xi)\\&\quad +\big(\chi_{\{b(\xi)<0\}}(\xi)e^{ita(\xi)}e^{(t-t_0)b(\xi)}\chi_{(0, \infty)}(t-t_0) \big) \hat{f}(\xi),
\end{split}
\end{equation}
for which we need to establish that
\begin{equation}\label{eq2.22}
\|Rf\|_{L_x^2} \leq C\|f\|_{L_x^2},
\end{equation}
with $C$ independent of $t_0$ and $m$.

But, looking at the multiplier in \eqref{eq2.21}, the estimate \eqref{eq2.22} holds true and this completes the proof.
\end{proof}

Our next result deals with the crucial upper estimate and reads as follows.

\begin{lemma}\label{lema-1}  There exists $k\in \Z^+ $ such that if $u\in C^{\infty}([0, 1]; C_0^{\infty}(\R))$, then for any $m\geq 1$, the following estimate holds:
\begin{equation}\label{eq2.2}
\begin{split}
\|e^{mx}u\|_{L_{xt}^8}+&\|e^{mx}\partial_xu\|_{L_{x}^{16}L_t^{16/5}}+\|e^{mx}\partial_x^2u\|_{L_{x}^{\infty}L_t^{2}}\\
&\leq Cm^{2k}\Big(\|J(e^{mx}u(\cdot, 0))\|_{L^2}+\|J(e^{mx}u(\cdot, 1))\|_{L^2}\Big)\\
&\quad + C \Big(\|e^{mx}Hu\|_{L_{xt}^{8/7}}+\|e^{mx}Hu\|_{L_x^{16/15}L_t^{16/11}}+\|e^{mx}Hu\|_{L_{x}^{1}L_t^2}\Big),
\end{split}
\end{equation}
 where $\widehat{Jg}(\xi):=(1+|\xi^2\|)^{1/2}\hat{g}(\xi)$ and $\|\cdot\|_{L_t^p}$ are restricted in $[0, 1]$.
\end{lemma}
\begin{proof}
 As noted in the beginning of this section, by Remark \ref{obspat}, we can suppose that $\beta>0$ and $|\alpha|/3=\beta$. Let us define
 \begin{equation}\label{eq2.3}
 v= e^{mx}u\in C^{\infty}([0, 1]; \s(\R)),
 \end{equation}
  then  the estimate \eqref{eq2.2} can be written as
  \begin{equation}\label{eq2.4}
  \begin{split}
  \|v\|_{L_{xt}^8}&+ \|e^{mx}\partial_xe^{-mx}v\|_{L_{x}^{16}L_t^{15/5}}+\|e^{mx}\partial_x^2e^{-mx}\|_{L_{x}^{\infty}L_t^{2}}\\
&\leq Cm^{2k}\Big(\|Jv(\cdot, 0)\|_{L^2}+\|Jv(\cdot, 1)\|_{L^2}\Big)\\
&\quad + C \Big(\|H_mv\|_{L_{xt}^{8/7}}+\|H_mv\|_{L_x^{16/15}L_t^{16/11}}+\|H_mv\|_{L_{x}^{1}L_t^2}\Big).
\end{split}
\end{equation}

The estimate \eqref{eq2.4} will hold true if we can prove the following set of estimates
\begin{equation}\label{eq2.5}
\|v\|_{L_{xt}^8} \leq C \Big( \|v(\cdot, 0)\|_{L^2}+\|v(\cdot, 1)\|_{L^2}\Big) + C\|H_mv\|_{L_{xt}^{8/7}},
\end{equation}
\begin{equation}\label{eq2.6}
\|e^{mx}\partial_xe^{-mx}v\|_{L_{x}^{16}L_t^{16/5}} \leq C m^k\Big( \|J^{1/2}v(\cdot, 0)\|_{L^2}+\|J^{1/2}v(\cdot, 1)\|_{L^2}\Big) + C\|H_mv\|_{L_{x}^{16/15}L_t^{16/11}}
\end{equation}
and
\begin{equation}\label{eq2.7}
\|e^{mx}\partial_x^2e^{-mx}v\|_{L_{x}^{\infty}L_t^{2}} \leq C m^{2k}\Big( J\|v(\cdot, 0)\|_{L^2}+\|Jv(\cdot, 1)\|_{L^2}\Big) + C\|H_mv\|_{L_{x}^{1}L_t^{2}}.
\end{equation}

\noindent
 We start by proving the estimate \eqref{eq2.5}:
As in Lemma \ref{lema-2}, it is enough to prove \eqref{eq2.5} for $v\in C^{\infty}([0, 1]:\mathcal{S}(\R))$ such that $\hat{v}(\xi, t) =0$ near $\xi_{\pm}$. Suppose that $f\in C^{\infty}([0, 1]:\mathcal{S}(\R))$ with $f(x,t)=0$ for $t$ near $0$ and $1$, so we can extend $f$ to 0 outside the strip $\R\times [0, 1]$. Also suppose that $\hat{f}(\xi, t) =0$ for $\xi$ near $\xi_{\pm}$ for all $t\in \R$. We will show that for the operator $T$ defined in \eqref{eq2.13}
 \begin{equation}\label{eq.estT1}
 \|Tf\|_{L_{xt}^8}\leq C\|f\|_{L_{xt}^{8/7}}
 \end{equation}
 and
 \begin{equation}\label{eq.estT2}
 \|Tf\|_{L_{xt}^8}\leq C\|f\|_{L_t^1L_{x}^{2}},
 \end{equation}
 for $f\in \mathcal{S}(\R^2)$ with $\hat{f}(\xi, t) =0$ for $\xi$ near $\xi_{\pm}$ for all $t\in \R$.

 The estimate \eqref{eq.estT1} is proved in \cite{CP2}. To get \eqref{eq.estT2} we restrict to consider $f(x,t) =f(x) \otimes \delta_{t_0}(t)$, and reduce the case to show that the operator $R$ defined in \eqref{eq2.21} satisfies
 \begin{equation}\label{eq.estT3}
 \|Rf\|_{L_{xt}^8}\leq C\|f\|_{L^{2}},
 \end{equation}
with $C$ independent of $m$ and $t_0$. But this is done in \cite{CP2}

Now we show that estimates \eqref{eq.estT1} and \eqref{eq.estT2} imply the estimate \eqref{eq2.5}. For this, consider
\begin{equation}\label{eq2.24}
v_{\varepsilon}(x,t) = \eta_{\varepsilon}(t)v(x,t), \qquad H_m(v_{\varepsilon})=\eta_{\varepsilon}'(t)v+\eta_{\varepsilon}H_m(v) =f_1(x,t)+f_2(x,t).
\end{equation}
Suppose,
\begin{equation}\label{eq2.25}
v_1(x,t)= Tf_1(x,t), \qquad v_2(x,t)= Tf_2(x,t),
\end{equation}
where both make sense because of our assumption on $v$. Then,
\begin{equation}\label{eq2.26}
v_{\varepsilon}(x,t) =v_1(x,t) +v_2(x,t),
\end{equation}
since both sides are in $l_{xt}^2$ and have the same Fourier transform. Hence, from \eqref{eq.estT1} and \eqref{eq.estT2} it follows that
\begin{equation}\label{eq2.26a}
\begin{split}
\|v_{\varepsilon} \|_{L_{xt}^8}&\leq \|v_1\|_{L_{xt}^8}+\|v_2\|_{L_{xt}^8}\leq C\|f_1\|_{L_t^1L_x^2} +C\|f_2\|_{L_{xt}^{8/7}}\\
&\leq C\|\eta_{\varepsilon}'(t)v\|_{L_t^1L_x^2} +C\|\eta_{\varepsilon}(t)H_mv\|_{L_{xt}^{8/7}}.
\end{split}
\end{equation}
Now, letting $\varepsilon\to 0$ we get the required estimate \eqref{eq2.5}.

\noindent
Next, we prove the estimate \eqref{eq2.7}:
As earlier, here too we make our usual assumptions on $\hat{v}(\xi, \tau)$. For $f \in \mathcal{S}(\R^2)$ with $\hat{f}(\xi, t) =0$ near $\xi_{\pm}$ for all $t\in \R$ we define
\begin{equation}\label{eq2.27}
\begin{split}
\widehat{T_2f}(\xi,\tau):&=(i\xi-m)^2\widehat{Tf}(\xi, \tau)\\
&= \frac{(i\xi-m)^2\hat{f}(\xi, \tau)}{i(\tau -\beta \xi^3 -\alpha \xi^2 +3\beta m^2\xi +\alpha m^2)-(\beta m^3-2\alpha m\xi -3\beta m\xi^2)}.
\end{split}
\end{equation}

Let \begin{equation}\label{eq2.28}
\tilde{T_2}f(x,t) = \chi_{[0,1]}T_2f(x,t).
\end{equation}
We will show that

\begin{equation}\label{eq2.29}
\|\tilde{T_2}f\|_{L_x^{\infty}L_t^2}\leq C\|f\|_{L_x^{1}L_t^2},
\end{equation}
\begin{equation}\label{eq2.30}
\|\tilde{T_2}f\|_{L_x^{\infty}L_t^2}\leq Cm^2\|Jf\|_{L_t^{1}L_x^2}.
\end{equation}

Before proving \eqref{eq2.29} and \eqref{eq2.30}, we show that these estimates imply \eqref{eq2.7}. Using the notations introduced in \eqref{eq2.24}, \eqref{eq2.25} and \eqref{eq2.26}, the estimates \eqref{eq2.29} and \eqref{eq2.30} yield
\begin{equation}
\begin{split}
\|(\partial_x-m)^2v_{\varepsilon}\|_{L_x^{\infty}L_t^2}&\leq \|\chi_{[0,1]}(\partial_x-m)^2v_{1}\|_{L_x^{\infty}L_t^2} +\|\chi_{[0,1]}(\partial_x-m)^2v_{2}\|_{L_x^{\infty}L_t^2}\\
&\leq \|\tilde{T_2}f_1\|_{L_x^{\infty}L_t^2} +\|\tilde{T_2}f_2\|_{L_x^{\infty}L_t^2}\\
&\leq Cm^2\|Jf_1\|_{L_t^{1}L_x^2} +C\|f_2\|_{L_x^{1}L_t^2}\\
&\leq Cm^2\|\eta_{\varepsilon}´(t)Jv\|_{L_t^{1}L_x^2} + C\|\eta_{\varepsilon}(t)H_mv\|_{L_x^{1}L_t^2}.
\end{split}
\end{equation}

Now in the limit as $\varepsilon \to 0$ we get \eqref{eq2.7}. So, to complete the proof of \eqref{eq2.7} it is enough to prove \eqref{eq2.29} and \eqref{eq2.30}.

With minor modification from the argument in \cite{CP2}, we get
\begin{equation}
\|T_2f\|_{L_x^{\infty}L_t^2} \leq C\|f\|_{L_x^{1}L_t^2},
\end{equation}
which in turn implies \eqref{eq2.29}.

Now we move to prove \eqref{eq2.30}. Let $\theta_r \in C_0^{\infty}(\R)$ with $\theta_r (x) =1$ for $|x|\leq 3r$ and $supp\; \theta_r\subset \{|x|\leq 4r\}$ and consider
\begin{equation}
\begin{split}
\widehat{T_2f}(\xi, \tau) & = \frac{\theta_m(\xi)(i\xi-m)^2\hat{f}(\xi, \tau)}{i(\tau -\beta \xi^3 -\alpha \xi^2 +3\beta m^2\xi +\alpha m^2)-(\beta m^3-2\alpha m\xi -3\beta m\xi^2)}\\
& \qquad +\frac{(1-\theta_m(\xi))(i\xi-m)^2\hat{f}(\xi, \tau)}{i(\tau -\beta \xi^3 -\alpha \xi^2 +3\beta m^2\xi +\alpha m^2)-(\beta m^3-2\alpha m\xi -3\beta m\xi^2)}\\
& =\widehat{T_{2,1}f}(\xi, \tau)+\widehat{T_{2,2}f}(\xi, \tau).
\end{split}
\end{equation}

Let $\tilde{T}_{2,1} :=\chi_{[0,1]}T_{2,1}$. From the Sobolev lemma we obtain
\begin{equation}\label{eq2.31}
\|\tilde{T}_{2,1}f\|_{L_x^{\infty}L_t^2} \leq C\|J\tilde{T}_{2,1}f\|_{L_x^{2}L_t^2}=C\|J\tilde{T}_{2,1}f\|_{L_t^{2}L_x^2}\leq C\|JT_{2,1}\|_{L_t^{\infty}L_x^2}.
\end{equation}

Now suppose,
$$ \hat{g_1}(\xi, \tau) =\theta_m(\xi)(1+|\xi|^2)^{1/2}(i\xi-m)^2\hat{f}(\xi,\tau),$$
so that $$JT_{2,1}f(x,t) = Tg_1(x,t)$$
and therefore from \eqref{eq2.14} and \eqref{eq2.31} it follows that
\begin{equation}\label{eq2.32}
\|\tilde{T}_{2,1}f\|_{L_x^{\infty}L_t^2} \leq C\|g_1\|_{L_t^{2}L_x^2}=Cm^2\|Jf\|_{L_t^{1}L_x^2}.
\end{equation}

To complete \eqref{eq2.30}, it is enough to prove
\begin{equation}\label{eq2.33}
\|T_{2,2}f\|_{L_x^{\infty}L_t^2} \leq C\|Jf\|_{L_t^{1}L_x^2}.
\end{equation}

Arguing  as  in \cite{E-KPV1}, the proof of this estimate can be reduced to consider functions of the form $f(x,t) =f(x)\otimes\delta_{t_0}(t)$; so that we just need to bound the operator
\begin{equation}
\widehat{R_{2,2}}(\xi, t) = (1-\theta_m(\xi))(i\xi-m)^2\chi_{\{b(\xi)<0\}}(\xi) e^{ita(\xi)}e^{(t-t_0)b(\xi)}\chi_{(0,\infty)}(t-t_0)\hat{f}(\xi),
\end{equation}
as
\begin{equation}
\|R_{2,2}f\|_{L_x^{\infty}L_t^2}\leq C\|Jf\|_{L_x^2},
\end{equation}
with $C$ independent of $m$ and $t_0$.

Let us write
\begin{equation}\label{eq2.34}
R_{2,2}f(x,t)= \int e^{ix\xi}(1-\theta_m(\xi))(i\xi-m)^2\chi_{\{b(\xi)<0\}}(\xi) e^{ita(\xi)}e^{(t-t_0)b(\xi)}\chi_{(0, \infty)}(t-t_0)\hat{f}(\xi)d\xi
\end{equation}
and recall that $a(\xi) = \beta\xi^3+\alpha\xi^2-3\beta m^2\xi-\alpha m^2$. Now, making change of variable $\lambda = a(\xi)$ we get $d\lambda = (3\beta\xi^2 +2\alpha\xi-3\beta m^2) d\xi$.

From the definition of $\theta_m(\cdot)$, the domain of integration in \eqref{eq2.34} is equal to $\{|\xi|\geq 3m\}$ where $|3\beta\xi^2 +2\alpha\xi-3\beta m^2|\cong |\xi|^2$ in fact if $\alpha \neq 0$, by Remark \ref{obspat}
\begin{align*}
|3\beta\xi^2 +2\alpha\xi-3\beta m^2| = &|\alpha|\, |\xi^2 \pm 2\xi- m^2|\\
\ge &|\alpha|( |\xi|^2- m^2-2|\xi|)\\
\ge &|\alpha|((8/9) |\xi|^2-2|\xi|)=|\alpha|\,|\xi|\{(8/9) |\xi|-2\}\geq |\alpha||\xi^2|/9,
\end{align*}
and the transformation is one-to-one since $a'(\xi)= |\alpha|( \xi^2 \pm 2 \xi-m^2) \gtrsim \xi^2$.

Thus we have $\xi = \xi(\lambda)$ and
\begin{equation}
\begin{split}
R_{2,2}f(x,t) &= \int e^{it\lambda}\frac{ e^{ix\xi}(1-\theta_m(\xi))(i\xi-m)^2}{3\beta\xi^2 +2\alpha\xi-3\beta m^2}\chi_{\{b(\xi)<0\}}(\xi) e^{(t-t_0)b(\xi)}\chi_{(0, \infty)}(t-t_0)\hat{f}(\xi)d\lambda\\
& = \int e^{it\lambda}\hat{g_ 2}(\lambda)\psi(\lambda, t)d\lambda,
\end{split}
\end{equation}
with
$$
\hat{g_2}(\lambda) = \frac{ e^{ix\xi}(1-\theta_m(\xi))(i\xi-m)^2}{3\beta\xi^2 +2\alpha\xi-3\beta m^2}\hat{f}(\xi),
$$
$$
\psi(\lambda, t) = \chi_{\{b(\xi)<0\}}(\xi) e^{(t-t_0)b(\xi)}\chi_{(0, \infty)}(t-t_0).
$$

Observe that,
$$|\psi(\lambda, t)| \leq C, \qquad \forall\; (\lambda, t)\in \R^2$$
and
$$\int |\partial_t\psi(\lambda, t)|dt\leq C\qquad \forall\; \lambda \in \R.$$

Therefore, using the result in \cite{CM1} and taking adjoint we get,
\begin{equation}
\begin{split}
\|\int e^{it\lambda}\hat{g_2}(\lambda)\psi(\lambda, t) d\lambda\|_{L_t^2}&\leq C\|\hat{g_2}\|_{L^2}\\
&\leq C\Big(\int\frac{|e^{x\xi}(1-\theta_m(\xi))(i\xi-m)^2\hat{f}(\xi)}{|3\beta\xi^2 +2\alpha\xi-3\beta m^2||3\beta\xi^2 +2\alpha\xi-3\beta m^2|}d\xi\Big)^{1/2}\\
&\leq C \Big(\int \frac{|1-\theta_m(\xi)|^2|\xi^2+m^2|^2|\hat{f}(\xi)|^2}{|3\beta\xi^2 +2\alpha\xi-3\beta m^2|}d\xi\Big)^{1/2}\\
&\leq C \|Jf\|_{L^2}
\end{split}
\end{equation}
which is \eqref{eq2.30}.

\noindent
Finally, we supply a proof of the estimate \eqref{eq2.6}:
At this point too, let us make the usual assumptions on $v$ and $\hat{v}$. For $f\in \mathcal{S}(\R^2)$ with $\hat{f}(\xi, t) =0$ near $\xi_{\pm}$ for all $t\in \R$, we define using \eqref{eq2.13}

\begin{equation}\label{eq2.35}
\widehat{T_1f}(\xi, \tau) = (i\xi-m)\widehat{Tf}(\xi, \tau) = \frac{(i\xi-m)\hat{f}(\xi, \tau)}{i(\tau -\beta \xi^3 -\alpha \xi^2 +3\beta m^2\xi +\alpha m^2)-(\beta m^3-2\alpha m\xi -3\beta m\xi^2)}.
\end{equation}

Now define,
\begin{equation}\label{eq2.36}
\tilde{T}_1f(x,t) = \chi_{[0, 1]}(t)T_1f(x,t).
\end{equation}

We claim that
\begin{equation}\label{eq2.37}
\|\tilde{T}_1f\|_{L_x^{16}L_t^{16/5}}\leq C\|f\|_{L_x^{16/15}L_t^{16/11}}
\end{equation}
and
\begin{equation}\label{eq2.38}
\|\tilde{T}_1f\|_{L_x^{16}L_t^{16/5}}\leq Cm\|J^{1/2}f\|_{L_t^{1}L_x^{2}}.
\end{equation}

As earlier, the estimate \eqref{eq2.6} easily follows from the estimates \eqref{eq2.37} and \eqref{eq2.38}. Let us recall,  in \cite{CP2} it was proved that
\begin{equation}
\|T_1f\|_{L_x^{16}L_t^{16/5}}\leq C\|f\|_{L_x^{16/15}L_t^{16/11}},
\end{equation}
which implies \eqref{eq2.37}. To obtain \eqref{eq2.38} we write $T_1$ in the following way

\begin{equation}
\begin{split}
\widehat{T_1f}(\xi, \tau) &= \frac{\theta_m(\xi)(i\xi-m)\hat{f}(\xi, \tau)}{i(\tau -\beta \xi^3 -\alpha \xi^2 +3\beta m^2\xi +\alpha m^2)-(\beta m^3-2\alpha m\xi -3\beta m\xi^2)}\\
&\qquad + \frac{(1-\theta_m(\xi))(i\xi-m)\hat{f}(\xi, \tau)}{i(\tau -\beta \xi^3 -\alpha \xi^2 +3\beta m^2\xi +\alpha m^2)-(\beta m^3-2\alpha m\xi -3\beta m\xi^2)}\\
&\qquad \widehat{T_{1,1}f}(\xi, \tau)+\widehat{T_{1,2}f}(\xi, \tau).
\end{split}
\end{equation}

Let $\tilde{T}_{1,1} =\chi_{[0,1]}(t)T_{1,1}$. Now from \eqref{eq2.32} we have
\begin{equation}
\|\tilde{T}_{2,1}f\|_{L_x^{\infty}L_t^2} \leq Cm^2\|Jf\|_{L_t^1L_x^2}
\end{equation}
and from \eqref{eq.estT2} we get
\begin{equation}
\|\tilde{T}_{0,1}f\|_{L_{xt}^8} \leq C\|f\|_{L_t^1L_x^2}.
\end{equation}

Hence, using the interpolation argument based on the Littlewood-Paley decomposition as in \cite{KPV-UC2} we obtain
\begin{equation}
\|\tilde{T}_{1,1}\|_{L_x^{16}L_t^{16/5}} \leq Cm \|J^{1/2}f\|_{L_t^1L_x^2}.
\end{equation}
Finally we interpolate between
\begin{equation}
\|\tilde{T}_{0,2}f\|_{L_{xt}^8} \leq C\|f\|_{L_t^1L_x^2},
\end{equation}
which follows from \eqref{eq.estT1}, with \eqref{eq2.32} to get
\begin{equation}
\|\tilde{T}_{1,2}f\|_{L_x^{16}L_t^{16/5}} \leq Cm \|J^{1/2}f\|_{L_t^1L_x^2},
\end{equation} and this yields \eqref{eq2.38}.
\end{proof}

In an analogous manner, as it has been worked out in \cite{E-KPV1}, the above result holds for a larger class of functions, for example:
$$u\in C([0, 1]; H^{k+3}(e^{\beta x}dx)\cap H^{k+3}(\R))\cap C^1([0, 1]; H^k(e^{\beta x}dx)\cap H^k(\R)),$$
with $k\in \Z$, $k\geq 1$ and for all $\beta >0$.

Now we want to extend the estimates in \eqref{eq2.2} in Lemma \ref{lema-1} to solutions with variable coefficients
\begin{equation}\label{eq2.45}
\pd_tu+i\alpha\pd_x^2u+\beta \pd_x^3u+ a_2(x,t)\pd_x^2u+a_1(x,t)\pd_xu+b_1(x,t)\pd_x\bar{u} +a_0(x,t) u +b_0(x,t)\bar{u} = g.
\end{equation}

Let us introduce the notation
\begin{equation}\label{eq2.46}
H_au := \pd_tu+i\alpha\pd_x^2u+\beta \pd_x^3u+ a_2(x,t)\pd_x^2u+a_1(x,t)\pd_xu+b_1(x,t)\pd_x\bar{u} +a_0(x,t) u +b_0(x,t)\bar{u},
\end{equation}
and suppose that multiplication by $a_0(x,t)$ and $b_0(x,t)$ map
\begin{equation}\label{eq2.48}
L_{xt}^8 \to L_{xt}^{8/7}, \qquad L_{xt}^8 \to L_{x}^{16/15}L_{t}^{16/11}, \qquad L_{xt}^8 \to L_{x}^{1}L_{t}^{2},
\end{equation}
multiplication by $a_1(x,t)$ and $b_1(x,t)$ map
\begin{equation}\label{eq2.49}
L_{x}^{16}L_t^{16/5} \to L_{xt}^{8/7}, \qquad L_{x}^{16}L_t^{16/5} \to L_{x}^{16/15}L_{t}^{16/11}, \qquad L_{x}^{16}L_t^{16/5} \to L_{x}^{1}L_{t}^{2}.
\end{equation}
and multiplication by $a_2(x,t)$ maps
\begin{equation}\label{eq2.499}
L_{x}^{\infty}L_t^{2} \to L_{xt}^{8/7}, \qquad L_{x}^{\infty}L_t^{2} \to L_{x}^{16/15}L_{t}^{16/11}, \qquad L_{x}^{\infty}L_t^{2} \to L_{x}^{1}L_{t}^{2}.
\end{equation}

To guarantee that the coefficients satisfy these conditions, it is enough to consider,
\begin{equation}\label{sp-as}
\begin{split}
&a_0, b_0 \in L_{x}^{16/13}L_t^{16/9}\cap L_{x}^{8/7}L_t^{8/3}\cap L_{x}^{16/15}L_t^{16/3},\\
&a_1, b_1 \in L_{xt}^{4/3}\cap L_{x}^{16/13}L_t^{16/9}\cap L_{x}^{8/7}L_t^{8/3}\\
&a_2 \in L_x^{8/7}L_t^{8/3}\cap L_x^{16/15}L_t^{16/3}\cap L_x^1L_t^{\infty},
\end{split}
\end{equation}

Also, if we assume that, if the coefficients satisfy
\begin{equation}\label{eq.cmx1}
\begin{split}
& a_0, b_0, a_1, b_1, a_2, \pd_x a_0, \pd_x b_0, \pd_x a_1, \pd_x b_1, \pd_x a_2, \pd_x^2 a_2, \pd_x^3 a_2, \pd_t a_2 \in  L^{\infty}(\R\times [0, 1]),\\
&a_2, \pd_t a_2 \in L_t^{\infty}([0, 1]; L_x^1(\R)).
\end{split}
\end{equation}
with small norms in \eqref{sp-as}, then Lemma \ref{lema-1} holds for $H_a$ instead of $H$. In fact we have the following result.

\begin{lemma}\label{lema-3}
Suppose that the coefficients $a_0, b_0, a_1, b_1, a_2$ satisfy \eqref{sp-as} and \eqref{eq.cmx1} with small norms in the spaces in \eqref{sp-as}. There exists $k\in \Z^+ $ such that if $u \in C^{\infty}([0, 1]; C_0^{\infty}(\R))$, then for any $m \geq 10 \|a_2\|_{L^{\infty}(\R\times [0, 1])}$
\begin{equation}\label{eq2.52}
\begin{split}
\|e^{mx}u\|_{L_{xt}^8}&+\|e^{mx}\partial_xu\|_{L_{x}^{16}L_t^{16/5}}+\|e^{mx}\partial_x^2u\|_{L_{x}^{\infty}L_t^{2}}\\
&\leq Cm^{2k}\Big(\|J(e^{mx}u(\cdot, 0))\|_{L^2}+\|J(e^{mx}u(\cdot, 1))\|_{L^2}\Big)\\
&\quad + C \Big(\|e^{mx}H_au\|_{L_{xt}^{8/7}}+\|e^{mx}H_au\|_{L_x^{16/15}L_t^{16/11}}+\|e^{mx}H_au\|_{L_{x}^{1}L_t^2}\Big),
\end{split}
\end{equation}
\end{lemma}
\begin{proof} Let us define
$$\tres f\tres_1 := \|e^{mx}f\|_{L_{xt}^8}+\|e^{mx}\partial_xf\|_{L_{x}^{16}L_t^{16/5}}+\|e^{mx}\partial_x^2f\|_{L_{x}^{\infty}L_t^{2}} $$

$$\tres f\tres_2 := \|f\|_{L_{xt}^{8/7}}+\|f\|_{L_x^{16/15}L_t^{16/11}}+\|f\|_{L_{x}^{1}L_t^2} $$

From Lemma \ref{lema-1} we have
\begin{equation}\label{eq2.612}
\begin{split}
\tres u\tres_1&\leq Cm^{2k}\Big(\|J(e^{mx}u(\cdot, 0))\|_{L^2}+\|J(e^{mx}u(\cdot, 1))\|_{L^2}\Big)
 + C \tres e^{mx}Hu\tres_2\\
 &\leq Cm^{2k}\Big(\|J(e^{mx}u(\cdot, 0))\|_{L^2}+\|J(e^{mx}u(\cdot, 1))\|_{L^2}\Big)
 + C \tres e^{mx}H_au\tres_2 \\
 &\quad +\tres e^{mx}(a_1\pd_xu+a_2\pd_x\bar{u}+a_3u+a_4\bar{u})\tres_2\\
 &\leq Cm^{2k}\Big(\|J(e^{mx}u(\cdot, 0))\|_{L^2}+\|J(e^{mx}u(\cdot, 1))\|_{L^2}\Big)
 + C \tres e^{mx}H_au\tres_2 +\frac12\tres u\tres_1,
\end{split}
\end{equation}
which gives the desired result.
\end{proof}

One can extend this result to a boarder class of solutions as in \cite{E-KPV1}.

\begin{theorem}\label{theo-1}
Let the coefficients $a_0, b_0, a_1, b_1, a_2$  satisfy the conditions in \eqref{sp-as} and \eqref{eq.cmx1}. If $u = u(x,t)$ is a solution of 
\begin{equation}\label{eq2.45-1}
\pd_tu+i\alpha\pd_x^2u+\beta \pd_x^3u+ a_2(x,t)\pd_x^2u+a_1(x,t)\pd_xu+b_1(x,t)\pd_x\bar{u} +a_0(x,t) u +b_0(x,t)\bar{u} = 0,
\end{equation}
 with $u\in C([0, 1];H^1(\R))$ satisfying that
$$u(\cdot, 0),\; u(\cdot, 1) \in H^1(e^{ax^{l}_+})$$
for some $l >1$ and $a>0$, then there exist $c_0$ and $R_0>0$ sufficiently large such that for $R\geq R_0$
\begin{equation}\label{eqth.1}
\|u\|_{L^2(\{R<x<R+1\}\times (0,1))} +\|\pd_xu\|_{L^2(\{R<x<R+1\}\times (0,1))}+\|\pd_x^2u\|_{L^2(\{R<x<R+1\}\times (0,1))}\leq c_0e^{-aR^{l}/4^{l}}.
\end{equation}
\end{theorem}
\begin{proof}
Choose $R$ so large that  in the $x$-interval $(R, \infty)$, the coefficients $a_0, b_0, a_1, b_1, a_2$ satisfy the conditions in \eqref{sp-as} and \eqref{eq.cmx1} with small norm in corresponding spaces in \eqref{sp-as}.

Let $\mu \in C^{\infty}(\R)$ with $\mu(x) = 0$ if $x< 1$ and $\mu(x) = 1$ in $x>2$.

For $\mu_{R}(x) = \mu(x/R)$, define
$$u_{R}(x,t) =\mu_{R}(x)u(x,t),$$
so that $u_{R}(x,t)$ satisfies the equation
\begin{equation}\label{eqx.1}
\pd_tu_R+\beta \pd_x^3u_R+i\alpha\pd_x^2u_R+a_2(x,t)\pd_x^2u_R+a_1(x,t)\pd_xu_R+b_1(x,t)\pd_x\bar{u_R} +a_0(x,t) u_R +b_0(x,t)\bar{u_R}= F_R
\end{equation}
where,
\begin{equation}\label{Fr}
\begin{split}
F_R &= \beta \frac1{R^3}\mu_R'''u +3\beta \frac1{R^2}\mu_R''\pd_xu+3\beta\frac1{R} \mu_R'\pd_x^2u +i\alpha \frac1{R^2}\mu_R''u +2i\alpha\frac1{R} \mu_R'\pd_xu\\
&\qquad +a_2(x,t)\frac1{R^2}\mu_R''u +2a_2(x,t)\frac1R\mu_R'\pd_x u+a_1(x,t)\frac1R\mu_R' u +b_1(x,t) \frac1R\mu_R'\bar{u}.
\end{split}
\end{equation}

Note that, ${\textrm{supp}}\; F_R \subset \{x: R< x < 2R\}$. Let us choose, $m = \frac a2 R^{l-1}$. Now, we cam use Lemma \ref{lema-3} to $u_R$ with
$$H_{a\tilde{\mu}_R} =\pd_t+\beta \pd_x^3+i\alpha\pd_x^2+\tilde{\mu}_Ra_2(x,t)\pd_x^2+\tilde{\mu}_Ra_1(x,t)\pd_x+\tilde{\mu}_Rb_1(x,t)\pd_x +\tilde{\mu}_Ra_0(x,t)  +\tilde{\mu}_Rb_0(x,t),$$
where $\tilde{\mu}_R(x)\mu_R(x)=\mu_R(x)$, which assures that the coefficients $\tilde{\mu}_R(x)a_j(x,t)$, $j= 0, 1,2$ and $\tilde{\mu}_R(x)b_j(x,t)$, $j= 0, 1$ have small norms in the corresponding spaces in \eqref{sp-as} for $R>R_0$. Therefore, applying \eqref{eq2.52} for $R$ large, we get
\begin{equation}\label{eq2.60}
\tres u_R\tres_1 \leq cm^{2k}\Big(\|J(e^{mx}u_R(\cdot, 0))\|_{L^2}+\|J(e^{mx}u_R(\cdot, 1))\|_{L^2}\Big)
+ \tres e^{mx}F_R\tres_ 2.
\end{equation}

With the argument similar to the one in \cite{E-KPV1}, the first two terms in the right hand side of \eqref{eq2.60} are bounded by $c_{a, l}$.

Now we move to bound the last term in \eqref{eq2.60}.

Recall that ${\textrm{supp}}\; F_R \subset \{x: R< x < 2R\}$. Now, the combination of H\"older and Minkowski´s integral inequality yield,
\begin{equation}\label{eq2.61}
\begin{split}
\tres e^{mx}F_R\tres_2& = \|e^{mx}F_R\|_{L_{xt}^{8/7}}+\|e^{mx}F_R\|_{L_x^{16/15}L_t^{16/11}}+\|e^{mx}F_R\|_{L_{x}^{1}L_t^2}\\
&\leq ce^{aR^{l-1}R}\|(|u|+|\pd_xu|+|\pd_x^2|)\chi_{\{x:R<x<2R\}}\|_{L_t^{\infty}L_x^2}\\
&\leq c'e^{aR^{l-1}R}.
\end{split}
\end{equation}

Hence, from \eqref{eq2.60} we obtain,
\begin{equation}\label{eq2.62}
\begin{split}
\|e^{mx}u\|_{L^8_{\{x>4R\}}L_t^8}+\|e^{mx}\partial_xu\|_{L_{{\{x>4R\}}}^{16}L_t^{16/5}}+\|e^{mx}\partial_x^2u\|_{L_{{\{x>4R\}}}^{\infty}L_t^{2}}\leq c_{a,l} +c'e^{aR^{l}}.
\end{split}
\end{equation}

Once again, using H\"older inequality in \eqref{eq2.62} we get, for sufficiently large $R$
\begin{equation}\label{eq2.63}
\begin{split}
\|u\|_{L^2_{(\{4R<x<4R+1\}\times (0,1))}}+\|\pd_xu\|_{L^2_{(\{4R<x<4R+1\}\times (0,1))}}+\|\pd_x^2u\|_{L^2_{(\{4R<x<4R+1\}\times (0,1))}}\leq c_{a,l}e^{-aR^{l}}.
\end{split}
\end{equation}

Replacing $4R$ by $R'$ we obtain,
\begin{equation}\label{eq2.64}
\begin{split}
\|u\|_{L^2_{(\{R'<x<R'+1\}\times (0,1))}}+\|\pd_xu\|_{L^2_{(\{R'<x<R'+1\}\times (0,1))}}+\|\pd_x^2u\|_{L^2_{(\{R'<x<R'+1\}\times (0,1))}}\leq c_{a,l}e^{-a(R'/4)^{l}},
\end{split}
\end{equation}
which yields the required estimate \eqref{eqth.1}.
\end{proof}

%%%%%%%%%%%%%%%%%%%%%%%%%%%%%%%%%%%%%%%%%%%%%%%%%%%%%%%%%%%%%%%%%%%%%%%%%%%
%%%%%%%%%%%%%%%%%%%%%%%%%%%%%%%%%%%%%%%%%%%%%
\secao{Lower  estimates}\label{prel}
%%%%%%%%%%%%%%%%%%%%%%%%%%%%%%%%%%%%%%%%%%%%%

This section is concerned with lower estimates that play fundamental role in the proof of the main result of this work. Let us begin with following lemma.
\begin{lemma}\label{mx1}
Assume that $\varphi$ is a real smooth function with compact support in $[0,1]$ and $\beta \neq 0$. Then, there exist $c>0$ and $M=M(\|\varphi'\|_{\infty};\|\varphi''\|_{\infty})>0$ such that the inequality
\begin{equation}\label{mx2}
\begin{split}
\dfrac{a^{5/2}}{R^3}&\left\|e^{a(\frac{x}{R}+\varphi(t))^2 }\left(\frac{x}{R}+\varphi(t)\right)^2 g\right\|_{L^2_{xt}}+\dfrac{a^{3/2}}{R^2}\left\|e^{a (\frac{x}{R}+\varphi(t))^2 }\left(\frac{x}{R}+\varphi(t)\right)^2 \partial_xg\right\|_{L^2_{xt}}\\
&+\dfrac{a^{1/2}}{R}\left\|e^{a(\frac{x}{R}+\varphi(t))^2 } \partial_x^2g\right\|_{L^2_{xt}}\le c(\beta) \left\|e^{a (\frac{x}{R}+\varphi(t))^2 }(\partial_t+i\alpha \partial_x^2+\beta \partial_x^3)g\right\|_{L^2_{xt}}
\end{split}
\end{equation}
holds, for $R \ge \alpha^2/\beta^{1/3}$, $a$ such that $a^2\ge MR^3$, and $g \in C_0^\infty(\R^2)$ supported in
\begin{equation*}
\left\{(x,t) \in \R^2: \left|\frac{x}{R}+\varphi(t)\right|\ge 1\right\}.
\end{equation*}
\end{lemma}
\begin{proof}
Initially we consider the case when $\beta=1$. In a similar way as in \cite{E-KPV1}, we define a function $f(x,t)=e^{a \theta(x,t)}g(x,t)$ with  $\theta(x,t)=(\frac{x}{R}+\varphi(t))^2$ and the expression
\begin{equation}\label{mx3}
e^{a \theta(x,t)}(\partial_t+i\alpha \partial_x^2+\partial_x^3)\left(e^{-a \theta(x,t)}f(x,t)\right)=(S_a+S_{a, \alpha})f+(A_a+A_{a, \alpha})f,
\end{equation}
where
\begin{equation*}
S_af= -3a (\theta_xf_x)_x-a^3 \theta_x^3f-a \theta_{xxx} f-a \theta_t f ; \quad S_{a, \alpha}f=-i \alpha a \theta_{xx} f-2i \alpha a \theta_x f_x,
\end{equation*}
and
\begin{equation*}
A_a f= f_t+f_{xxx}+3 a^2 \theta_x^2 f_x+3 a^2 \theta_x \theta_{xx}f; \quad A_{a, \alpha}f=i \alpha a^2 \theta_x^2 f +i \alpha f_{xx}.
\end{equation*}

We have $S_a^*=S_a$, $S_{a, \alpha}^*=S_{a, \alpha}$, $A_a^*=-A_a $ and $A_{a, \alpha}^*=-A_{a, \alpha}$, and therefore,
\begin{equation}\label{mx4}
\begin{split}
&\left\|e^{a (\frac{x}{R}+\varphi(t))^2 }(\partial_t+i\alpha \partial_x^2+\partial_x^3)g\right\|_{L^2_{xt}}^2= \left\|(S_a+S_{a, \alpha})f+(A_a+A_{a, \alpha})f \right\|_{L^2_{xt}}^2 \\
& \ge \left\langle \left\{(S_a A_a- A_a S_a)+(S_a A_{a, \alpha}- A_{a, \alpha}S_a)
+( S_{a, \alpha} A_a- A_a S_{a, \alpha})+( S_{a, \alpha} A_{a, \alpha}- A_{a, \alpha} S_{a, \alpha})\right\}f,f\right\rangle.
\end{split}
\end{equation}

We find that
\begin{equation}\label{mx5}
\begin{split}
&(S_a A_a- A_a S_a)f= \,\,[S_a;A_a]f \\
&=9a(\theta_{xx}f_{xx})_{xx}+((6a\theta_{xt}-18a^3\theta_x^2\theta_{xx})f_x)_x+(-3a^3\theta_{xx}^3+a \theta_{tt}+6a^3\theta_{x}^2\theta_{xt}+9a^5\theta_{x}^4\theta_{xx})f,
\end{split}
\end{equation}
\begin{equation}\label{mx6}
[S_a;A_{a, \alpha}]f= [S_{a, \alpha};A_{a}]f
= \,\,i6a\alpha\theta_{xx}f_{xxx}-i6a^3\alpha\theta_x\theta_{xx}^2f-i6a^3\alpha\theta_x^2\theta_{xx}f_x+i2a\alpha\theta_{xt}f_x,
\end{equation}
and
\begin{align}\label{mx7}
[S_{a, \alpha};A_{a, \alpha}]f
= \,\,-4a\alpha^2\theta_{xx}f_{xx}.
\end{align}

In \cite{E-KPV1} it was proved that, if $a^2 \ge ( \|\varphi'\|_{\infty}+\|\varphi''\|_{\infty}^{1/2}+1)R^3$, then
\begin{equation}\label{mx8}
\begin{split}
\left\langle [S_a; A_a]f,f\right\rangle &\ge \frac{18 a}{R^2}\int\!\!\!\int \left|f_{xx}\right|^2dxdt+\frac{132 a^3}{R^4}\int\int \left(\frac{x}{R}+\varphi(t)\right)^2\left|f_{x}\right|^2dxdt\\
& \qquad+\frac{216 a^5}{R^6}\int\int \left(\frac{x}{R}+\varphi(t)\right)^4\left|f\right|^2dxdt.
\end{split}
\end{equation}

From \eqref{mx6} and \eqref{mx8} one has that
\begin{equation}\label{mx9}
\begin{split}
&\left\langle [S_a;A_{a, \alpha}]f+ [S_{a, \alpha};A_{a}]f+[S_{a, \alpha};A_{a, \alpha}]f,f\right\rangle \\ &\qquad\quad=i12a\alpha\int\!\!\!\int\theta_{xx}f_{xxx}\bar{f}dxdt -i12a^3\alpha\int\!\!\!\int\theta_x\theta_{xx}^2|f|^2dxdt \\
&\qquad\quad\quad -i12a^3\alpha\int\!\!\!\int\theta_x^2\theta_{xx}f_x\bar{f}dxdt +i4a\alpha\int\!\!\!\int\theta_{xt}f_x\bar{f}dxdt +4a\alpha^2\int\!\!\!\int\theta_{xx}|f_{x}|^2dxdt\\
&\qquad\quad=:iJ_1+iJ_2+iJ_3+iJ_4+J_5.
\end{split}
\end{equation}

Integrating by parts, we observe that $iJ_1, iJ_2+iJ_3, i J_4 \in \R$. Since  $\theta(x,t)=(\frac{x}{R}+\varphi(t))^2$ we have that
\begin{equation}\label{mx101}
|iJ_1| \le \dfrac{24a |\alpha|}{R^2}\left|\int\!\!\!\int f_{xx}\bar{f_{x}}dxdt\right|\le \dfrac{12a }{R^2}\int\!\!\!\int |f_{xx}|^2dxdt+\dfrac{12\alpha^2 a}{R^2}\int\!\!\!\int|f_{x}|^2 dxdt.
\end{equation}

Similarly
\begin{equation}\label{mx10}
\begin{split}
|iJ_2+iJ_3|& =|\Im J_3| \le |J_3| \le 96\int\!\!\!\int\left(\dfrac{a^{5/2}(\frac{x}{R}+\varphi(t))^2|f|}{R^3}\right)\left( \dfrac{|\alpha|a^{1/2}|\bar{f}_{x}|}{R}\right) dxdt\\
&\le \frac{48 a^5}{R^6}\int\int \left(\frac{x}{R}+\varphi(t)\right)^4\left|f\right|^2dxdt+\dfrac{48\alpha^2 a}{R^2}\int\!\!\!\int|f_{x}|^2 dxdt.
\end{split}
\end{equation}

As $a^2 \ge \left\|\varphi'\right\|_{L^{\infty}}R^3$ and $(\frac{x}{R}+\varphi(t))^4>1$, on the support of $f$, we get
\begin{equation}\label{mx11}
\begin{split}
|iJ_4| & \le 8\int\!\!\!\int \left(a^{1/2}|\varphi' f|\right)\left(\dfrac{a^{1/2} |\alpha| |\bar{f_{x}|}}{R}\right)dxdt\\
&\le 4a \int\!\!\!\int |\varphi' f|^2dxdt+\dfrac{4\alpha^2 a}{R^2}\int\!\!\!\int|f_{x}|^2 dxdt\\
&\le  \frac{4 a^5}{R^6}\int\int \left(\frac{x}{R}+\varphi(t)\right)^4\left|f\right|^2dxdt+\dfrac{4\alpha^2 a}{R^2}\int\!\!\!\int|f_{x}|^2 dxdt.
\end{split}
\end{equation}

Now $R \ge \alpha^2$ and $a^2 \ge R^3$ imply that
\begin{equation}\label{mx12}
\dfrac{\alpha^2 a}{R^2}\int\!\!\!\int|f_{x}|^2 dxdt \le \frac{a^3}{R^4}\int\int \left(\frac{x}{R}+\varphi(t)\right)^2\left|f_{x}\right|^2dxdt.
\end{equation}

Combining \eqref{mx101}-\eqref{mx12} we obtain
\begin{equation*}
\begin{split}
iJ_1+iJ_2+iJ_3+iJ_4+J_5 & \ge \frac{-12 a}{R^2}\int\!\!\!\int \left|f_{xx}\right|^2dxdt-\frac{64 a^3}{R^4}\int\int \left(\frac{x}{R}+\varphi(t)\right)^2\left|f_{x}\right|^2dxdt\\
&\qquad -\frac{52 a^5}{R^6}\int\int \left(\frac{x}{R}+\varphi(t)\right)^4\left|f\right|^2dxdt.
\end{split}
\end{equation*}
This inequality, \eqref{mx4} and \eqref{mx9} yield
\begin{equation*}
\begin{split}
\left\|e^{a (\frac{x}{R}+\varphi(t))^2 }(\partial_t+i\alpha \partial_x^2+\partial_x^3)g\right\|_{L^2_{xt}}^2 &\ge \left\langle [S_a;A_{a}]f+ [S_a;A_{a, \alpha}]f+ [S_{a, \alpha};A_{a}]f+[S_{a, \alpha};A_{a, \alpha}]f,f\right\rangle\\
& \ge  \frac{6 a}{R^2}\int\!\!\!\int \left|f_{xx}\right|^2dxdt+\frac{68 a^3}{R^4}\int\int \left(\frac{x}{R}+\varphi(t)\right)^2\left|f_{x}\right|^2dxdt\\
&\qquad +\frac{164 a^5}{R^6}\int\int \left(\frac{x}{R}+\varphi(t)\right)^4\left|f\right|^2dxdt,
\end{split}
\end{equation*}
which concludes the proof of the Lemma when $\beta=1$.

Now if $\beta>0$, $\beta \neq 1$ (see Remark \ref{obspat}) we use the case $\beta=1$ with $\alpha:= \alpha / \beta^{2/3}$ and $g(x,t)=g(\beta^{1/3}x,t)$. Finally, we perform a change of variable $x:=\beta^{1/3}x$ to obtain \eqref{mx2}.

\end{proof}
In an analogous manner as in \cite{E-KPV1}, we have the following result.
\begin{lemma}\label{lemmx1}
Assume that $\varphi$ is a real smooth function with compact support in $[0,1]$ and that $a_0, a_1, b_0, b_1$ are complex functions in $L^{\infty}(\R^2)$.
Then there exist $c>0$, $$R_0=R_0(\|\varphi'\|_{\infty};\|\varphi''\|_{\infty};\|a_0\|_{\infty};\|a_1\|_{\infty})>1, \,\,\textrm{and}\,\, M=M(\|\varphi'\|_{\infty};\|\varphi''\|_{\infty})>0$$ such that the inequality
\begin{align}\label{lemmx2}
&\dfrac{a^{5/2}}{R^3}\left\|e^{a(\frac{x}{R}+\varphi(t))^2 }\left(\frac{x}{R}+\varphi(t)\right)^2 g\right\|_{L^2_{xt}}+\dfrac{a^{3/2}}{R^2}\left\|e^{a (\frac{x}{R}+\varphi(t))^2 }\left(\frac{x}{R}+\varphi(t)\right)^2 \partial_xg\right\|_{L^2_{xt}}\nonumber\\
&\le c \left\|e^{a (\frac{x}{R}+\varphi(t))^2 }(g_t+i\alpha g_{xx}+\beta g_{xxx}+a_1(x,t)g_x+b_1(x,t)\bar{g}_x+a_0(x,t)g+b_0(x,t)\bar{g})\right\|_{L^2_{xt}}
\end{align}
holds, for $R \ge R_0+ \alpha^2$, $a$ such that $a\ge MR^{3/2}$, and $g \in C_0^\infty(\R^2)$ supported in
$$
\left\{(x,t) \in \R^2: \left|\frac{x}{R}+\varphi(t)\right|\ge 1\right\}.
$$
\end{lemma}
\begin{theorem}\label{teomx1}
Let $u \in C([0,1]; H^3(\R))$ be a solution of
\begin{equation}\label{mx15}
u_t+i\alpha u_{xx}+\beta u_{xxx}+ a_2(x,t)u_{xx}+ a_1(x,t)u_{x}+b_1(x,t)\bar{u}_{x}+a_0(x,t)u+b_0(x,t)\bar{u}=0,
\end{equation}
with $b_0, b_1, a_0, a_1, a_2, (a_2)_x, (a_2)_{xx} \in L^{\infty}(\R^2)$ and $a_2, (a_2)_t \in L_t^{\infty}(\R: L_x^1(\R))$. If
\begin{equation}\label{mx16}
\int_{\R}\!\int_0^1(|u|^2+|u_x|^2+|u_{xx}|^2)(x,t)dxdt \le A^2,
\end{equation}
and
\begin{equation}\label{mx17}
\int_{3/8}^{5/8}\!\int_0^1u^2(x,t)dxdt \ge 1,
\end{equation}
then there exist constants $R_0,c_0,c_1>0$ depending on
$$
A, \|b_0\|_{\infty}, \|b_1\|_{\infty}, \|a_0\|_{\infty}, \|a_1\|_{\infty}, \|a_2\|_{\infty}, \|\partial_x a_2\|_{\infty}, \|\partial_x^2 a_2\|_{\infty}, \|a_2\|_{L_t^\infty L_x^1}, \|\partial_t a_2\|_{L_t^\infty L_x^1}
$$
such that for $R \ge R_0$
\begin{equation}\label{1mx17}
\delta(R)=\delta_u(R)=\left(\int_0^1\!\int_{R-1}^R(|u|^2+|u_x|^2+|u_{xx}|^2)(x,t)dxdt \right)^{1/2}\ge c_0e^{-c_1R^{3/2}}.
\end{equation}
\end{theorem}
\begin{proof}
Considering the gauge transformation
\begin{equation}\label{2mx17}
v(x,t)=u(x,t)e^{1/(3\beta)\int_0^x a_2(s,t)ds},
\end{equation}
the equation for $v=v(x,t)$ can be written as
\begin{align}\label{mx13}
v_t &- \left(\frac{1}{3\beta}\int_0^x \partial_t a_2(s,t)ds\right) v+i\alpha\left(v_{xx}-\frac{2}{3\beta}a_2v_x+\left(-\frac{1}{3\beta}\partial_x a_2+\frac{1}{9\beta^2}a_2^2\right)v \right)\nonumber\\
&+\beta \left( v_{xxx}-\frac{a_2}{\beta }v_{xx}+\left(-\frac{1}{3\beta}\partial_x a_2+\frac{1}{9\beta^2}a_2^2\right) v_x+\left(-\frac{a_2^3}{27\beta^3}+\frac{a_2}{3\beta^2}\partial_x a_2-\frac{1}{3\beta}\partial_x^2 a_2\right)v\right)\nonumber\\
&+a_2 v_{xx}-\frac{2a_2^2}{3\beta}v_x+\left(-\frac{a_2 \partial_x a_2}{3\beta}+\frac{a_2^3}{9\beta^2}\right)v+a_1v_x-\frac{a_2 a_1}{3\beta}v +a_0v-\frac{\bar{a}_2 b_1}{3\beta}\bar{v}_x+\left(b_0-\frac{\bar{a}_2 b_1}{3\beta}\right)\bar{v}\nonumber\\
= & \,v_t+ i \alpha v_{xx}+\beta v_{xxx}+ \tilde{a_1}(x,t)v_x+\tilde{a_0}(x,t)v+ \tilde{b_1}(x,t)\bar{v}_x+\tilde{b_0}(x,t)\bar{v}=0.
\end{align}
where $\tilde{a_0}, \tilde{a_1}, \tilde{b_0}, \tilde{b_1}$ are complex functions in $L^{\infty}(\R^2)$.

As in \cite{E-KPV1}, we define the functions
$\theta_R(x)=1$ if $x<R-1$, $\theta_R(x)=0$ if $x>R$,  $\mu(x)=1$ if $x>2$, $\mu(x)=0$ if $x<1$ and $\varphi(t)=3$ if $t \in [3/8, 5/8]$, $\varphi(t)=0$ if $t \in [0, 1/4]\cup [3/4,1]$,  $0 \le \theta_R, \mu \le 1$, $\theta_R, \mu \in C^{\infty}(\R)$, $0\le \varphi \le 3$ and $\varphi \in C_0^{\infty}(\R)$ and the function
$$
g(x,t)=\theta_R(x) \mu \left(\frac{x}{R}+\varphi(t)\right)v(x,t), \quad (x,t) \in \R \times [0,1],
$$
so that $g$ has support on $(-2R,R)\times(0,1)$ and can be assumed to satisfy the hypothesis of Lemma \ref{mx1}.

Using \eqref{mx15} one has that
\begin{equation}\label{mx18}
\begin{split}
&g_t+i \alpha g_{xx}+ \beta g_{xxx}+\tilde{a}_1g_x+\tilde{b}_1 \bar{g}_x+\tilde{a}_0g+\tilde{b}_0\bar{g} \\
& =\mu\left(\frac{x}{R}+\varphi(t)\right)\left(i \alpha \theta_R^{(2)}v+2i \alpha \theta_R^{(1)}v_x+\beta \theta_R^{(3)}v+3\theta_R^{(2)}v_x+3 \beta \theta_R^{(1)}v_{xx}+\tilde{a}_1 \theta_R^{(1)}v+\tilde{b}_1 \theta_R^{(1)}\bar{v}\right)\\
&\quad+\theta_R(x)\left(\mu^{(1)}\left(\varphi^{(1)}+\frac{\tilde{a}_1}{R}\right)v +i \alpha \frac{\mu^{(2)}}{R^2}v+2i \alpha \frac{\mu^{(1)}}{R}v_x+\beta\frac{\mu^{(3)}}{R^3}v+3\beta\frac{\mu^{(2)}}{R^2}v_x+3\beta\frac{\mu^{(1)}}{R}v_{xx}\right)\\
&\quad+ 2i \alpha \theta_R^{(1)}\frac{\mu^{(1)}}{R}v+3\beta \theta_R^{(2)}\frac{\mu^{(1)}}{R}v+3\beta \theta_R^{(1)}\frac{\mu^{(2)}}{R^2}v+6\beta \theta_R^{(1)}\frac{\mu^{(1)}}{R}v_x+\tilde{b}_1\theta_R\frac{\mu^{(1)}}{R}\bar{v}.
\end{split}
\end{equation}

The remaining part of the proof follows as in \cite{E-KPV1}. In fact, using the definitions of $\theta_R, \mu$, $\varphi$ and \eqref{mx17} we get
\begin{equation}\label{mx19}
\dfrac{a^{5/2}}{R^3}\left\|e^{a(\frac{x}{R}+\varphi(t))^2 }\left(\frac{x}{R}+\varphi(t)\right)^2 g\right\|_{L^2_{xt}}\ge c\dfrac{a^{5/2}}{R^3} e^{9a}.
\end{equation}

On the other hand, we observe that the first term in the right-hand side of \eqref{mx18} is supported in $[R-1,R]\times [0,1]$ where $e^{a(x/R +\varphi(t))^2}\le e^{16a}$, and in the remaining terms in the right-hand side of \eqref{mx18} we have $e^{a(x/R +\varphi(t))^2}\le e^{4a}$. Thus \eqref{lemmx2} and \eqref{mx16} imply that for $a \ge M_1R^{3/2}$ ($M_1$ as in Lemma \ref{lemmx1}),
\begin{align}\label{mx20}
\dfrac{a^{5/2}}{R^3}&\left\|e^{a(\frac{x}{R}+\varphi(t))^2 }\left(\frac{x}{R}+\varphi(t)\right)^2 g\right\|_{L^2_{xt}}\nonumber\\
\le & c \left\|e^{a (\frac{x}{R}+\varphi(t))^2 }(g_t+i\alpha g_{xx}+\beta g_{xxx}+a_1(x,t)g_x+b_1(x,t)\bar{g}_x+a_0(x,t)g+b_0(x,t)\bar{g})\right\|_{L^2_{xt}}\nonumber\\
\le &  c_1 e^{16a} \delta_v(R)+c_1 e^{4a}A.
\end{align}

Combining \eqref{mx19} and \eqref{mx20} it follows that
$$
c\dfrac{a^{5/2}}{R^3}e^{9a} \le c_1 e^{16a} \delta_v(R)+c_1 e^{4a}A, \quad \forall a \ge M_1R^{3/2}.
$$
In particular, for $a=M_1R^{3/2}$ with $R$ sufficiently large we obtain
$$
\delta_v(R)\ge c_0e^{-c_1R^{3/2}}.
$$
By the hypothesis on the coefficients $a_0, a_1, a_2$ and the definitions \eqref{1mx17}, \eqref{2mx17} we conclude that
$$
\delta_u(R) \sim \delta_v(R)\ge c_0e^{-c_1R^{3/2}}.
$$
\end{proof}

%%%%%%%%%%%%%%%%%%%%%%%%%%%%%%%%%%%%%%%%%%%%%%%%%%%%%%%%%%%%%%%%%%%%
\secao{Proof of the Main Results}\label{main-results}
%%%%%%%%%%%%%%%%%%%%%%%%%%%%%%%%%%%%%%%%%%%%%%%%%%%%%%%%%%%%%%%%%%%%

This section is devoted to provide proofs of the main results of this work. First, let us begin with the proof of the Theorem \ref{ucp5}.\\

\noindent
{\bf Proof of Theorem \ref{ucp5}.}
If $u \neq 0$, we can suppose that $u$ satisfies the hypothesis of Theorem \ref{teomx1} and therefore
\begin{equation}\label{mx21}
\delta_u(R) \ge c_0e^{-c_1R^{3/2}},
\end{equation}
and apply Theorem \ref{theo-1} with $l=3/2$, $a \gg 8c_1$, $c_1$ as above we have
$$
\delta_u(R) \le c e^{-aR^{3/2}/8},
$$
which is a contradiction with \eqref{mx21} for $R$ sufficiently large.\hfill$\Box$

Now we are position to supply proof if the first main result of this work.\\

\noindent
{\bf Proof of Theorem \ref{ucp3}.}
Let $u_1$, $u_2$ be strong solutions of the equation \eqref{0y0}, then their difference $w=u_1-u_2$ satisfies the following equation
\begin{equation}\label{mx100}
\partial_t w+i \alpha \partial_x^2w +\beta\partial_x^3 w+\delta a_1(x,t)\partial_x w+\epsilon b_1(x,t)\partial_x \bar{w}+a_0(x,t)w +b_0(x,t)\bar{w}=0,
\end{equation}
where $a_1(x,t)=|u_1|^2$, $b_1(x,t)=u_1^2$, $a_0(x,t)=i \gamma(|u_1|^2+|u_2|^2)+\delta \bar{u}_2 \partial_x u_2+\epsilon (u_1+u_2)
\partial_x\bar{u}_2$ and $b_0(x,t)=i \gamma u_1 u_2+\delta u_1 \partial_x u_2$.

To conclude the proof of the theorem, it is   sufficient to prove that $w, a_0, a_1, b_0, b_1$ satisfy 
the hypotheses of Theorem \ref{ucp5}. As in \cite{E-KPV1}, this is a consequence of
the estimates
\begin{equation}\label{mx24}
\|vu\|_{L_x^pL_t^p} \le
\|v\|_{L_x^{\infty}L_t^{\infty}}\|u\|_{L_x^pL_t^p}.
\end{equation}

For the sake of completeness, we present the proofs of the estimates correcting some mistakes present in \cite{E-KPV1}. 

We need show that
$$
a_0,b_0 \in L_{xt}^{4/3}\cap L_x^{16/13}L_t^{16/9}\cap
L_x^{8/7}L_t^{8/3}, \quad a_1,b_1  \in 
L_x^{16/13}L_t^{16/9}\cap L_x^{8/7}L_t^{8/3}\cap
L_x^{16/15}L_t^{16/3}.
$$

We will prove the estimates only for $b_0$ and $b_1$, because those for $a_0$ and $a_1$ are similar. Using  the hypothesis
\begin{equation}\label{mx23}  \quad u_j
\in C([0,1]:H^3\cap L^2(|x|^2 dx)), \quad j=1,2,
\end{equation}
we have (see \cite{E-KPV1})
\begin{align}
|x|u_j, |x|^{2/3}(u_j)_x, |x|^{1/3}(u_j)_{xx}, (u_j)_{xxx} \in
L^{\infty} ([0,1], L_x^2), \quad j=1,2\label{mx27} \\ u_j,
|x|^{2/3} u_j \in L^{\infty} ([0,1], L_x^{\infty}), \quad
j=1,2.\label{mx25}
\end{align}
Thus, \eqref{mx24}, \eqref{mx25} and Holder´s inequality yield
\begin{equation*}
\|u_1 u_2\|_{L_{xt}^{4/3}} \le c \|u_1\|_{L_{xt}^{\infty}}
\sup_{t\in[0,1]}\|\langle x \rangle^{1/2^+}u_2\|_{L_x^2}\le c
\|u_1\|_{L_{xt}^{\infty}} \sup_{t\in[0,1]}(\|u_2\|_{L_x^2}+\| x
u_2\|_{L_x^2}).
\end{equation*}

Similarly
\begin{equation*}
\|u_1 (u_2)_x\|_{L_{xt}^{4/3}} \le c \|u_1\|_{L_{xt}^{\infty}}
\sup_{t\in[0,1]}\|\langle x \rangle^{1/2^+}(u_2)_x\|_{L_x^2}\le c
\|u_1\|_{L_{xt}^{\infty}} \sup_{t\in[0,1]}(\|u_2\|_{L_x^2}+\|
|x|^{2/3} (u_2)_x\|_{L_x^2}).
\end{equation*}

Now we will prove that $b_0 \in L_x^{16/13}L_t^{16/9}$. We have,
\begin{equation*}
\begin{split}
\|u_1 u_2\|_{L_x^{16/13}L_t^{16/9}} \le & c
\|u_1\|_{L_{xt}^{\infty}} \left(\int\!\!\frac{\langle x \rangle^p}{\langle x \rangle^p}\int \left( |u_2|^{16/9}dt\right)^{9/13} dx\right)^{13/16}, \quad 4/13 <p <15/13,\\
\le & c
\|u_1\|_{L_{xt}^{\infty}} \left(\int\!\!\int\langle x
\rangle^{13p/9} |u_2|^{16/9}dxdt\right)^{9/16}, \\
\le&  c \|u_1\|_{L_{xt}^{\infty}} \left(\int \|\langle x
\rangle u_2\|_{L_x^2}^{16/9}dt\right)^{9/16}\\
\le & c \|u_1\|_{L_{xt}^{\infty}}
\sup_{t\in[0,1]}(\|u_2\|_{L_x^2}+\| x u_2\|_{L_x^2}).
\end{split}
\end{equation*}

Analogously, using $4/13<p<35/39$, we get
$$ \|u_1 (u_2)_x\|_{L_x^{16/13}L_t^{16/9}} \le c
\|u_1\|_{L_{xt}^{\infty}} \sup_{t\in[0,1]}(\|u_2\|_{L_x^2}+\|
|x|^{2/3} (u_2)_x\|_{L_x^2}).$$ 

Now we will prove that $b_0 \in L_x^{8/7}L_t^{8/3}$. Similarly as in \cite{E-KPV1}, we get
\begin{align*}
\|u_1 u_2\|_{L_x^{8/7}L_t^{8/3}} \le & c
\|u_1\|_{L_{xt}^{\infty}} \left(\int\!\!\int\langle x
\rangle^{4^+/3} |u_2|^{8/3}dxdt\right)^{3/8}\\
\le & c \|u_1\|_{L_{xt}^{\infty}}\, \|\langle
x\rangle^{3/2\epsilon}
u_2\|_{L_{xt}^{\infty}}^{1/4}\sup_{t\in[0,1]}\|\langle
x\rangle^{2/3} u_2\|_{L_{xt}^{2}}^{3/4},\quad 0<\epsilon \ll 1,
\end{align*}
and we have similar estimate  for
$\|u_1(u_2)_x\|_{L_x^{8/7}L_t^{8/3}}$.

Finally, it is sufficient to prove that $b_1=u_1^2 \in
L_x^{16/15}L_t^{16/3}$. In fact (see \cite{E-KPV1})
$$
\|u_1^2\|_{L_x^{16/15}L_t^{16/3}}\le c
\|u_1\|_{L_{xt}^{\infty}}\,\|\langle x\rangle
u_1\|_{L_{xt}^{2}}^{3/8} \|\langle x\rangle^{2/3}
u_1\|_{L_{xt}^{\infty}}^{5/8}.
$$
Using \eqref{mx23}-\eqref{mx25} we conclude the proof of the
theorem.\hfill$\Box$

In what follows, we provide the proof of the second main result about the decay property of the solution to the Hirota equation.\\

\noindent
{\bf Proof of Theorem \ref{teoairy}.} The proof of this theorem is very similar to the proof of Theorem 1.4 in \cite{E-KPV1}. For the sake of clarity, we provide a brief idea pointing out the differences that arise in our case.

Let  $\psi \in C_0^{\infty}(\mathbb{R})$, $\psi \ge 0$, $\textrm{supp} \,\psi \subset(-\delta, \delta)$, $\delta \in (0, 1/8)$ and $\int \psi(x)dx=1$.

We consider the IVP
\begin{equation}\label{lnpt12}
\begin{cases}
u_t + i\alpha u_{xx} + \beta u_{xxx}+F(u)= 0, \quad (x,t) \in \mathbb{R}\times [0, \Delta T]\\
u(x, 0) = u_0(x)=\varepsilon V(1)\phi= \varepsilon S_1 *\psi,
\end{cases}
\end{equation}
where $F(u)=i\gamma|u|^{2}u+\delta
|u|^{2}\partial_{x}u+\epsilon u^{2}\partial_{x}\overline{u}$,\,\, and $\varepsilon, \Delta T$ are sufficiently small.

Without loss of generality we can suppose $\beta=1$ (Remark \ref{obspat}). 

Let $V(t) u_0$ be the solution of IVP \eqref{lnpt12} when $\beta=1$ and let $V_{\alpha }(t) u_0$ be the solution of IVP \eqref{lnpt12} when $\beta=0$.

Let us  consider $\phi=V_{-\alpha} \psi$, thus $\psi= V_{\alpha} \phi$,
then
$$
V(t)\phi=S_t *V_{\alpha}\phi= S_t *\psi,
$$
where $S_t *f(x)= \frac1{\sqrt[3]{3t}}Ai(\frac{\cdot}{\sqrt[3]{3t}})*f(x)$ and $Ai(x)$ is the Airy function.

 The solution to the IVP \eqref{lnpt12} (see \cite{E-KPV1} and \cite{KPV1}) is obtained by iterating
$$
\Phi(u^n)(t)=u^{n+1}(t)=\varepsilon S_{1+t}*\psi+\int_0^t S_{t-t'}*F(u^n)(t')dt',
$$
$n=1,2,\cdots$ in the ball
\begin{equation}\label{norma0}
|\!|\!| w |\!|\!|_{T,s,k}\le 2\varepsilon(\|S_1*\psi\|_{H^s}+\|x^k S_1*\psi\|_{L^2}),
\end{equation}
where
\begin{align}\label{norma1}
|\!|\!| w |\!|\!|_{T,s,k}=\sup_{[0,T]}(\|w(t)\|_{H^s}+\|x^k w(t)\|_{L^2})+\|w(t)\|_{L_x^2L_t^{\infty}([0,T])}+\|\partial_x^{s+1}w(t)\|_{L_x^{\infty}L_t^2([0,T])}.
\end{align} 
The sequence $\{u^n\}$ converges in the norm given by \eqref{norma1}, for $T>0$ sufficiently small, inside the ball defined in \eqref{norma0}.

Using the induction principle, the integral equation and properties of $S_{t}*\psi$ (Airy function), for $t\in [1, 1+\Delta T]$, $\Delta T>0$ small enough (see \cite{E-KPV1}), we obtain
\begin{equation}\label{lnpt5}
|F^n(x,t)| \le c\varepsilon^3\,
\begin{cases}
e^{-x^{3/2}},\quad \textrm{if}\quad x>1/2,\\
1,\quad \textrm{if}\quad |x| \le 1/2,\\
1/(1+x^2)^{2k},\quad \textrm{if}\quad x \in \mathbb{R}.
\end{cases}
\end{equation}
This inequality, properties of Airy function, a limit process and the same argument as in \cite{E-KPV1} for $\varepsilon$ sufficiently small, yield the desired result.\hfill$\Box$

%%%%%%%%%%%%%%%%%%%%%%%%%%%%%%%%%%%%%%%%%%%%%

\end{document}